\documentclass[10pt]{article}
\font\smallit=cmti10
\font\smalltt=cmtt10

\usepackage{amssymb,latexsym,amsmath,epsfig,amsthm,tabularx,chngpage} 

\makeatletter

\renewcommand\section{\@startsection {section}{1}{\z@}
{-30pt \@plus -1ex \@minus -.2ex}
{2.3ex \@plus.2ex}
{\normalfont\normalsize\bfseries}}

\renewcommand\subsection{\@startsection{subsection}{2}{\z@}
{-3.25ex\@plus -1ex \@minus -.2ex}
{1.5ex \@plus .2ex}
{\normalfont\normalsize\bfseries}}

\renewcommand{\@seccntformat}[1]{\csname the#1\endcsname. }

\makeatother

\newtheorem{theorem}{Theorem}

\newcommand{\beq}{\begin{equation}}
\newcommand{\eeq}{\end{equation}}

\def\({\left(}
\def\){\right)}

\begin{document}

\begin{center}
\uppercase{\bf A geometric approach to integer factorization}
\vskip 20pt
{\bf Dmitry I. Khomovsky}\\
{\smallit Lomonosov Moscow State University}\\
{\tt khomovskij@physics.msu.ru}
\end{center}
\vskip 30pt

\centerline{\smallit Received: , Revised: , Accepted: , Published: } 
\vskip 30pt

\centerline{\bf Abstract}

\noindent
We give a geometric approach to integer factorization. This approach is based on  special approximations of segments of the curve that is represented by $y=n/x$, where $n$ is \mbox{the integer} whose factorization we need.

\pagestyle{myheadings}
\markright{\smalltt \hfill}
\thispagestyle{empty}
\baselineskip=12.875pt
\vskip 30pt

\section{Introduction}
Let $n=x y$ be an odd integer and $x, y$ be its nontrivial factors.
It is known that $n$  can be represented as the difference of two squares:
\beq\label{0}
n=\left(\frac{x+y}{2}\right)^2-\left(\frac{x-y}{2}\right)^2.
\eeq
This property is used in Fermat's factorization method, which is based on searching for the representation of an odd integer as $n=a^2-b^2$. To find such a representation we need to take values of $a\geq \lceil\sqrt{n}\,\rceil$ and determine whether $a^2-n$ is a perfect square. Fermat's method is the most efficient when there is a factor near $\sqrt{n}$. It is used in the so-called multiplier improvement that was applied by Sherman Lehman in \cite{0}. The main idea of this improvement consists in searching for a multiplier $r< n^{1/3}$ such that $r n$ has a factor near $\sqrt{r n}$. Lehman's algorithm has worst-case running time $O(n^{1/3})$.
Although the above  methods are rarely  used for practical purposes, the ideas underlying them are a part of more efficient methods for factoring integers. We only give some references that allow the reader to become familiar with existing methods (see \cite{1,2,3,4}).

In this paper, we propose one approach to factorization of integers which has a simple geometric interpretation.
This approach allows us to look at Fermat's factorization method and similar methods from a different angle.
\section{\bf{The main theorem}}
Let $\mathcal{C}$ be a plane curve defined by an equation $f(x,y)=0$ in Cartesian coordinates, where $f: \mathbb{R}^2\to \mathbb{R}$. The following theorem gives a method for finding solutions of the Diophantine equation $f(x,y)=0$, in other words integral points on $\mathcal{C}$.
\begin{theorem}\label{T1}
Let $f$, $g$ be functions from $\mathbb{R}^2$ to $\mathbb{R}$,
and $g$ such that $g(x,y)\in \mathbb{Z}$ if $x, y\in \mathbb{Z}$. If for any integer $k$ from $a\leq k\leq b$, where $a, b\in \mathbb{Z}$, the system of equations $f(x,y)=0$, $g(x,y)=k$  does not have an integer solution, then the Diophantine equation $f(x,y)=0$ does not have solutions on $A=\{(x,y):a\leq g(x,y)\leq b\}$.
\end{theorem}
\begin{proof}
Suppose that the equation $f(x,y)=0$ has an integer solution $(x_0, y_0)$  that belongs to the set $A$. Since $x_0, y_0$ are integers, $g(x_0, y_0)$ is also an integer, moreover, we have $a\leq g(x_0, y_0)\leq b$. But this contradicts  the theorem conditions.
\end{proof}
The geometric meaning of the above result is revealed in the following reasoning.
To find integral points on a smooth  segment of the curve $\mathcal{C}$ or to show that they do not exist, we locally approximate $\mathcal{C}$ by another curve with the equation $g(x,y)=0$, and  $g$ must be such that $g(x,y)\in \mathbb{Z}$ if $x, y\in \mathbb{Z}$. After this, we look for integer solutions of the system $f(x,y)=0$, $g(x,y)=k$ for $a\leq k\leq b$, where $a, b$ are chosen so that the segment considered belongs to $A$.
\section{\bf{Factoring integers}}
We consider the curve $\mathcal{C}$ represented by the explicit equation $y=n/x$ with $x>0$. This curve is related to the problem of factoring $n$. At $x=\sqrt{n}$ the tangent to $\mathcal{C}$ is represented by $y=-x+2\sqrt{n}$. If we take $y=-x+\lfloor2\sqrt{n}\rfloor$, then the corresponding line lies under the positive branch of $\mathcal{C}$, and the function $g(x,y)=x+y-\lfloor2\sqrt{n}\rfloor$ satisfies the conditions of Theorem $1$. Therefore, if we consider  the system $y=n/x$, $y=-x+\lfloor2\sqrt{n}\rfloor+k$ for $0\leq k\leq b$ and if $n$ has a divisor $p$ in the interval \beq\label{1}\frac{1}{2}\left(\lfloor2\sqrt{n}\rfloor+b+1-\sqrt{(\lfloor2\sqrt{n}\rfloor+b+1)^2-4n}\right)<p\leq\sqrt{n},\eeq
then we can find this divisor. To determine whether the above system  has integer solutions it is sufficient to check whether $(\lfloor2\sqrt{n}\rfloor+k)^2-4n$ is a perfect square. Since $n$ is odd,  we may only consider the system for $k$ such  that $\lfloor2\sqrt{n}\rfloor+k$ is even. Also, if some additional information on divisors of $n$ is known, it can be used to reduce the number of values of $k$ that need to be checked in the interval $0\leq k\leq b$. For example, divisors of the Fermat numbers $F_m=2^{2^m}+1$  ($m\geq 2$) are of the form $r\cdot 2^{m + 2} + 1$, which was established by Euler and Lucas. Then for $F_m=p q$ we can show that $p+q\equiv 2\bmod {2^{2(m+2)}}$. Thus, we need to check only values of $k$ such that $k\equiv 2-\lfloor2\surd{F_m}\rfloor\bmod 2^{2(m+2)}$.

Note that the above-described factorization method is exactly Fermat's method.
Indeed, let us consider the term $(x+y)/2$ on the right-hand side of the formula $(\ref{0})$. Its smallest possible value is $\sqrt{n}$, since $y=n/x$. Therefore, the choice of integer values of this term is equivalent to the choice of $\ell$ in $y=-x+2\lfloor\sqrt{n}\rfloor+2\ell$. The following illustration reveals the geometrical meaning of the method.
\begin{center}
 \includegraphics[scale=0.5]{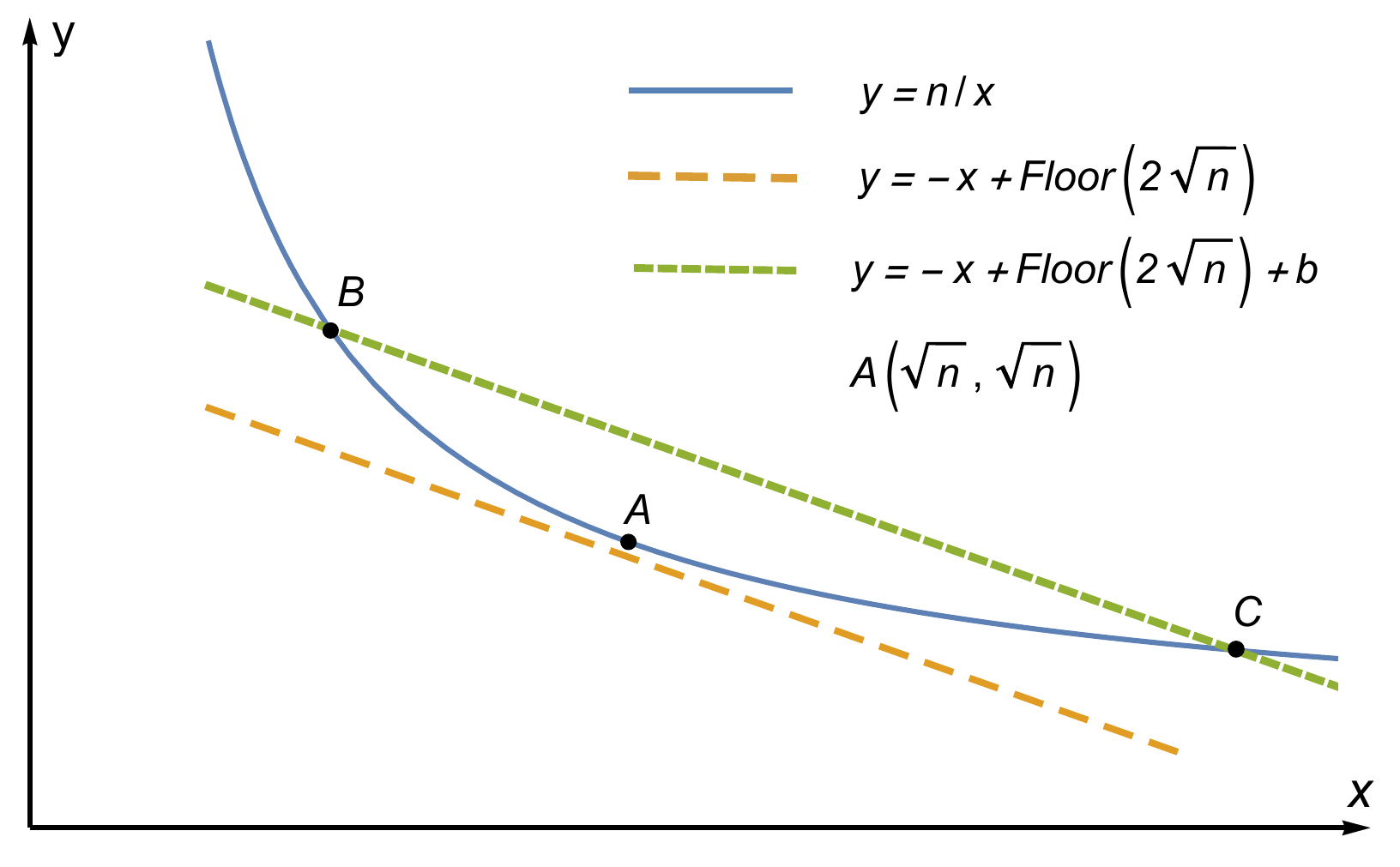}

 { Figure 1:  Graphic illustration of Fermat's method.}
\end{center}
{\bf Remark.} We can consider an approximation of segments of  $\mathcal{C}$ by osculating circles. At the point $(x_0,n/x_0)$, the osculating circle is given by
\beq\label{2}
\left(x-\frac{n^2+3x_0^4}{2x_0^3}\right)^2+\left(y-\frac{3n^2+x_0^4}{2 n x_0}\right)^2=\frac{(x_0^4+n^2)^3}{4n^2x_0^6}.
\eeq
At $x_0=\sqrt{n}$ we have $(x-2\sqrt{n})^2+(y-2\sqrt{n})^2=2n$. The solution of the system\footnote{For simplicity, we can consider the system $(x-2\sqrt{n})^2+(y-2\sqrt{n})^2=2n+k, y=n/x$.} $(x-[2\sqrt{n}\,])^2+(y-[2\sqrt{n}\,])^2-2n=\lfloor 2(\sqrt{n}-[2\sqrt{n}\,])^2-2n\rfloor+k, y=n/x$ shows that for large $n$ and small $b$  if we examine the system for all integers $k$ from $1\leq k\leq b$, then the interval in which we are looking for divisors has the length approximately equal to $2(b n)^{1/4}$. This is no better than $2(b^2 n)^{1/4}$ in Fermat's method.
\subsection{Using the Taylor series}
The first-degree Taylor polynomial of the function $n/x$ at $x=\sqrt{n/s}$ $(s\in\mathbb{Z^+})$  is $-sx+2\sqrt{sn}$.
Since the coefficients of the polynomial $g(x,y)=y+sx-\lfloor2\sqrt{sn}\rfloor$ are integers, then $g(x,y)$ can be used to search for divisors of  $n$ in a neighborhood of $\sqrt{n/s}$ (see Theorem $1$). It can be shown that for large $n$ we can check for divisors the interval with the length approximately equal to $2(b^2 n/s^3)^{1/4}$  by making $b$ steps. Here, by one step we mean checking the existence of integer solutions of $n/x+sx-\lfloor2\sqrt{sn}\rfloor=k$. We  see that for $s\geq n^{1/3}$, i.e., for $x\leq n^{1/3}$, the first-order approximation is inefficient.

Now we consider the second-degree Taylor polynomial of the function $n/x$ at $x=(n/s)^{1/3}$ $(s\in\mathbb{Z^+})$. It is equal to $s x^2 - 3 (s^2 n)^{1/3} x + 3 (s n^2)^{1/3}$. The polynomial $g(x,y)=y-s x^2 + \big[ 3 (s^2 n)^{1/3}\big] x - \big[ 3 (s n^2)^{1/3}\big]$ with integer coefficients
can be used to search for divisors of  $n$ in a neighborhood of $(n/s)^{1/3}$. We need to use the equation
\beq\label{3}n/x-s x^2 + \big[ 3 (s^2 n)^{1/3}\big] x - \big[ 3 (s n^2)^{1/3}\big]=\big\lfloor\big[ 3 (s^2 n)^{1/3}\big](n/s)^{1/3}-\big[ 3 (s n^2)^{1/3}\big]\big\rfloor+k.\eeq
If for integers $k=1, 2,\ldots, b$ we answer the question whether there exists an integer solution of $(\ref{3})$, then the  length of the  interval that we have checked for divisors is $(b^3 n/s^4)^{1/9}$. For $-b+1\leq k\leq 0$ the  length of the  interval is also $(b^3 n/s^4)^{1/9}$, but this interval is to the right of the point $(n/s)^{1/3}$ in contrast to the previous case. The following figure is an illustration of the search for divisors of $n$ in a neighborhood of $n^{1/3}$.
\begin{center}
 \includegraphics[scale=0.56]{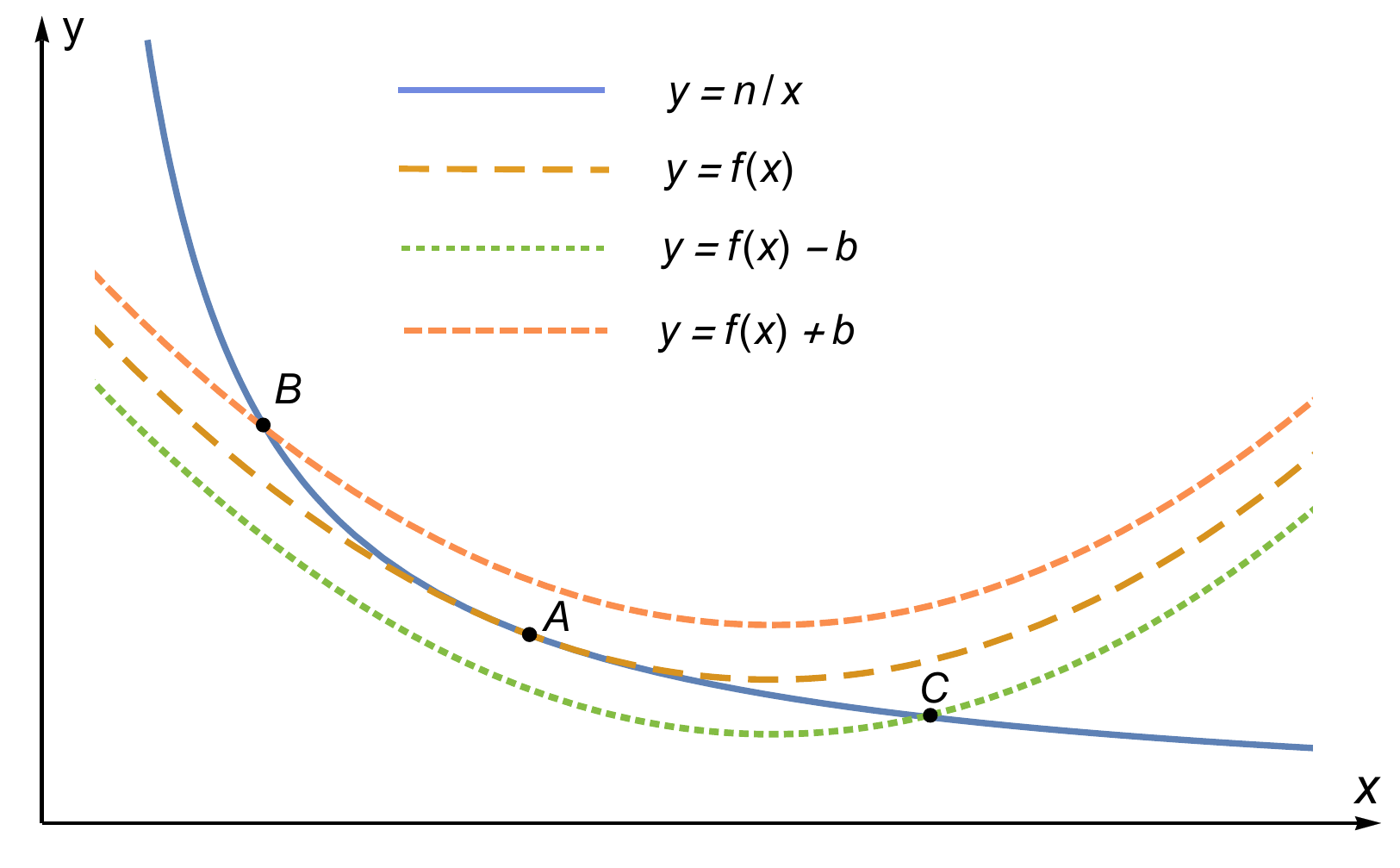}

 { Figure 2:  The second-order approximation of $n/x$ at $x=n^{1/3}$.\\
  $f(x)=x^2 - \big[ 3  n^{1/3}\big] x + \big[ 3 n^{2/3}\big]+\big\lfloor\big[ 3  n^{1/3}\big]n^{1/3}-\big[ 3 n^{2/3}\big]\big\rfloor$.}
 \end{center}
As can be seen from the estimate of the length of the interval, the second-order approximation is inefficient for $s\geq n^{1/4}$, which is equivalent to $x\leq n^{1/4}$.

\noindent {\bf Remark.} The cubic equation $ax^3+bx^2+cx+d=0$ \mbox{with integer coefficients} has one integer root and two complex conjugate roots only if $-\Delta$ is a perfect square, where $\Delta=18abcd-4b^3d+b^2c^2-4ac^3-27a^2d^2$ is  the discriminant of the equation. From Figure $2$ we see that this is realized, when we solve the equation $(\ref{3})$. Thus, in order to answer the question whether $(\ref{3})$ has an integer root, we should check its discriminant  before solving it.

 By analogy with the previous reasoning, we have
\beq\label{4}n/x+s x^3 - \big[ 4 (s^3 n)^{1/4}\big] x^2+ \big[ 6 (s^2 n^2)^{1/4}\big]x-\big[ 4 (s n^3)^{1/4}\big]=C+k,\eeq
where 
\beq\label{4.1}C=\big\lfloor2(s n^3)^{1/4} - \big[ 4 (s^3 n)^{1/4}\big] (n/s)^{1/2}+ \big[ 6 (s^2 n^2)^{1/4}\big](n/s)^{1/4}-\big[ 4 (s n^3)^{1/4}\big]\big\rfloor,\eeq
for divisors in a neighborhood of $(n/s)^{1/4}$.
The estimate of the  length of the interval  corresponding to $b$ steps is $2(b^4 n/s^5)^{1/16}$. Finally, the equation for divisors of $n$ in a neighborhood of $(n/s)^{1/m}$ is:
\beq\label{5}n/x+(-1)^m s x^{m-1} + \sum_{i=1}^{m-1}(-1)^{m-i}\bigg[\binom{m}{i}(s^{m-i}n^i)^{1/m}\bigg]x^{m-1-i}=C+k,\eeq
where
\beq\label{6}C=\big\lfloor\left(1+(-1)^m\right) (s n^{m-1})^{1/m} + \sum_{i=1}^{m-1}(-1)^{m-i}\bigg[\binom{m}{i}(s^{m-i}n^i)^{1/m}\bigg]({n/s})^{(m-1-i)/m} \big\rfloor.\eeq

So far we have considered the case $s\in\mathbb{Z^+}$, but $(\ref{5})$ can be used for rational $s=h/t$ if $k$ is replaced by $k/t$. We put $s=1/t$ $(t\in\mathbb{Z^+})$ in $(\ref{3})$ and replace $k$ by $k/t$, then we obtain the equation for divisors of  $n$ in a neighborhood of $(n t)^{1/3}$. In this case the  length of the interval  corresponding to $b$ steps is $(b^3 t n)^{1/9}$. On the other hand, if we use the first-order approximation of $n/x$ at $x=(n t)^{1/3}$, then the  length of the interval is $2(bt)^{1/2}$. Comparing the obtained estimates with each other, we can answer the question: at what values of $x$ is the first-order approximation more efficient? To answer it, we need to solve the following inequality $2(bt)^{1/2}>(b^3 t n)^{1/9}$. If we put $b=1$, i.e., we make only one check, then we get $t>(n/2^9)^{2/7}$. Finally, we have $x>(n/4)^{3/7}$. The result should be considered approximate.

\noindent {\bf Remark.} If we use $(\ref{5})$ at $x=n^{1/m}$, the interval for divisors is proportional to $n^{1/m^2}$. The following question arises: is there such an approximation of $\mathcal{C}$ at an arbitrary point $x=n^e$ $(0<e<1/2)$ which gives  the length of the interval that is proportional to $x^{e}$?
\subsection{Lehman-like methods}
Let $r$ be a positive integer. The search for divisors of $r n$ in a neighborhood of $\sqrt{r n}$ can be carried out using the equation $r n/x+x-\lfloor2\sqrt{r n}\rfloor=k$ $(k\in \mathbb{Z^+})$. It has the solutions $x=\left(A\pm\sqrt{A^2-4r n}\right)/2$, where $A=\lfloor2\sqrt{r n}\rfloor+k$, which are integer if $A^2-4r n$ is a perfect square. As it was shown, the length of the interval corresponding to $b$ steps $(k=1,2,\ldots,b)$ is equal to $2(b^2 r n)^{1/4}$. We see that if $b=\left\lceil n^{1/6}/(4\sqrt{r})\right\rceil$, then the length is greater than or equal to $n^{1/3}$. 
The relationship between the above and Lehman's method becomes clear if one looks at the version of Lehman's algorithm in \cite{5}.

Let us agree that if we apply the multiplier improvement using an approximation of $r n/x$ at $x=u(r n)$, where $u:\mathbb{R}\to \mathbb{R}$, then we will call it  a divisors trap related  to the function $u$.
In these terms, Lehman's method is based on the trap related to $u(\omega)=\sqrt{\omega}$. We can modify this method by using another trap with the function $v(\omega)=\omega^{1/3}$. So if large $n$ has a divisor near $n^{1/3}$, then it will fall into the latter trap with a smaller value of $r$. In other words, the more traps the smaller the value of the multiplier at which a divisor of composite $n$ will be found. This idea and formulas $(\ref{5}), (\ref{6})$ can be used to obtain efficient factorization algorithms.


\smallskip
\noindent{\bf Acknowledgments.}
I thank my school math teacher Ms. Nadezhda P. Vlasova.

\medskip

\end{document}